\documentclass[a4paper,10pt]{amsart}
\usepackage{amsmath,amssymb,graphicx,url}

\newtheorem{theorem}{Theorem}[section]
\newtheorem{prop}[theorem]{Proposition}
\newtheorem{coro}[theorem]{Corollary}
\newtheorem{lemma}[theorem]{Lemma}
\newtheorem{fact}[theorem]{Fact}

\theoremstyle{definition}

\begin{document}
\title[$(-2,3,2s+1)$-Pretzel knot and $\mathbb{R}$-covered foliations]{A good presentation of $(-2,3,2s+1)$-type Pretzel knot group and $\mathbb{R}$-covered foliation}
\author{Yasuharu NAKAE}
\address{Graduate School of Engineering and Resource Science, Akita University,
1-1 Tegata Gakuen-machi, Akita city, Akita, 010-8502, Japan}
\email{nakae@math.akita-u.ac.jp}
\thanks{
typeset by \AmS-\LaTeX
}
\begin{abstract}
Let $K_s$ be a $(-2,3,2s+1)$-type Pretzel knot ($s\geqq 3$)
and $E_{K_s}(p/q)$ be a closed manifold obtained by Dehn surgery along $K_s$ with a slope $p/q$.
We prove that if $q>0$, $p/q\geqq 4s+7$ and $p$ is odd, then $E_{K_s}(p/q)$ cannot contain an $\mathbb{R}$-covered foliation.
This result is an extended theorem of a part of works of Jinha Jun for $(-2,3,7)$-Pretzel knot.
\end{abstract}
\maketitle

\section{Introduction}

In this paper, we will discuss non-existence of $\mathbb{R}$-covered foliations
on a closed $3$-manifold obtained by Dehn surgery along some class of a Pretzel knot.

A codimension one, transversely oriented foliation $\mathcal{F}$ on a closed $3$-manifold $M$ is called a Reebless foliation
if $\mathcal{F}$ does not contain a Reeb component.
By the theorems of Novikov~\cite{No}, Rosenberg~\cite{Ro}, and Palmeira~\cite{Pa},
if $M$ is not homeomorphic to $S^2\times S^1$ and contains a Reebless foliation,
then $M$ has properties that
the fundamental group of $M$ is infinite,
the universal cover $\widetilde{M}$ is homeomorphic to $\mathbb{R}^3$
and all leaves of its lifted foliation $\widetilde{\mathcal{F}}$ on $\widetilde{M}$
are homeomorphic to a plane.
In this case we can consider a quotient space $\mathcal{T}=\widetilde{M}/\widetilde{\mathcal{F}}$,
and $\mathcal{T}$ is called a leaf space of $\mathcal{F}$.
The leaf space $\mathcal{T}$ becomes a simply connected $1$-manifold,
but it might be a non-Hausdorff space.
If the leaf space is homeomorphic to $\mathbb{R}$, $\mathcal{F}$ is called an $\mathbb{R}$-covered foliation.
The fundamental group $\pi_1(M)$ of $M$ acts on the universal cover $\widetilde{M}$ as deck transformations.
Since this action maps a leaf of $\widetilde{\mathcal{F}}$ to a leaf, 
it induces an action of $\pi_1(M)$ on the leaf space $\mathcal{T}$.
In fact, it is known that the action has no global fixed point and it acts on $\mathcal{T}$ as a homeomorphism 
(for basic definitions and properties of a foliation, see \cite{CC1}, and for properties of a Reebless foliation and its leaf space, see ~\cite[{\it Chapter 9} and {\it Appendix D}]{CC2}).

A closed manifold with a finite fundamental group cannot contain a Reebless foliation by the above property,
but many people conjectured that all closed hyperbolic $3$-manifolds could contain Reebless foliations.
In \cite{RSS}, Roberts, Shareshian and Stein negatively answered this conjecture as follows:

\begin{theorem}\label{ThmRSS}{\rm (R.\,Roberts, J.\,Shareshian, M.\,Stein\,\cite[Theorem A]{RSS})}
There exist infinitely many closed orientable hyperbolic $3$-manifolds
which do not contain a Reebless foliation.
\end{theorem}

There is an analogous concept of a Reebless foliation, an essential lamination~\cite{GO}.
If a closed $3$-manifold contains an essential lamination,
it has topological properties which are similar to those with a Reebless foliation.  
In \cite{Fen}, Fenley showed that there exist infinitely many closed hyperbolic $3$-manifolds
which do not admit essential laminations. 

In \cite{Jun}, Jun applied the methods used in the proof of Theorem \ref{ThmRSS}
to prove the following theorem:

\begin{theorem}\label{ThmJun1}{\rm (J.\,Jun\,\cite[Theorem 1]{Jun})} 
Let $K$ be a $(-2,3,7)$-Pretzel knot in $S^3$
and $E_K(p/q)$ be a closed manifold obtained by Dehn surgery along $K$ with slope $p/q$.
If $p/q>18$, $p$ is odd and $p/q\neq 37/2$,
then $E_K(p/q)$ does not contain a Reebless foliation.
\end{theorem}

In the case of an $\mathbb{R}$-covered foliation,
Jun also proved the following theorem:

\begin{theorem}\label{ThmJun}{\rm (J.\,Jun\,\cite[Theorem 2]{Jun})}
$K$ and $E_K(p/q)$ are the same as {\rm Theorem~\ref{ThmJun1}}.
If $p/q\geqq 10$ and $p$ is odd,
then $E_K(p/q)$ does not contain an $\mathbb{R}$-covered foliation.
\end{theorem}

In this paper,
we shall prove a theorem which is an extension of Theorem~\ref{ThmJun}
to the case of $(-2,3,2s+1)$-type Pretzel knot ($s\geqq 3$) as follows.

\begin{theorem}\label{maintheorem}{\rm (Main Theorem)}
Let $K_s$ be a $(-2,3,2s+1)$-type Pretzel knot in $S^3$ {\rm ($s\geqq 3$)}.
If $q>0$, $p/q\geqq 4s+7$ and $p$ is odd,
then $E_{K_s}(p/q)$ does not contain an $\mathbb{R}$-covered foliation.
\end{theorem}

These theorems are proved by a similar strategy as follows.
Let $M$ be a closed $3$-manifold and $\mathcal{F}$ be a Reebless foliation in $M$.
Then, as we stated before,
the fundamental group $\pi_1(M)$ acts on the leaf space $\mathcal{T}$ of $\mathcal{F}$
as an orientation preserving homeomorphism
which has no global fixed point.
By the theorem of Palmeira,
for two Reebless foliations $\mathcal{F}$ and $\mathcal{F}'$
there is a diffeomorphism $f:M\to M$ which maps $\mathcal{F}$ to $\mathcal{F}$
if and only if
there is a diffeomorphism $\tilde{f}:\mathcal{T} \to \mathcal{T}'$,
where $\mathcal{T}$ and $\mathcal{T}'$ are the leaf spaces of $\mathcal{F}$ and $\mathcal{F}'$ respectively.
Therefore,
for any simply connected $1$-manifold $\mathcal{T}$,
if there exists a point of $\mathcal{T}$ which is fixed by any action of $\pi_1(M)$
then $M$ cannot contain a Reebless foliation.

In order to use above method to prove our main theorem,
we will need an explicit presentation of the fundamental group.
Moreover, it is better for proving our theorem that its presentation has simpler form
because our investigation of existence of a global fixed point becomes easy
if its presentation has fewer generators.

This paper is constructed as follows.
In Section 2, we shall present a method of an explicit calculation to obtain a good presentation
of the fundamental group
of the exterior of the knot $K_s$ and the elements which represent a meridian and a longitude.
Its presentation has two generators and one relator, and this property comes from the fact that
the knot $K_s$ is a tunnel number one knot.
In Section 3, by using the good presentation obtained in Section 2
we shall prove the main theorem by comparing to procedures of the proof of Theorem~\ref{ThmJun} and Theorem~\ref{ThmRSS}.
In Section 4, we will discuss topics and problems related to our theorem,
especially a left-orderable group and its properties.

\section{an explicit presentation of the fundamental group}

As mentioned in Section 1, we need a good presentation of a fundamental group and
a meridian-longitude pair in order to make a proof of our theorem easier.
In the proof of \cite{Jun}, Jun uses the presentation of a knot group of $(-2,3,7)$-pretzel knot
which obtained by the computer program, SnapPea~\cite{W}.
Let $K_s$ be a $(-2,3,2s+1)$-type Pretzel knot in $S^3$.
In order to obtain a good presentation of the knot group of $K_s$ and its meridian-longitude pair,
we take the following procedure.

We first notice that $K_s$ is a tunnel number one knot for all $s\geqq 3$
by the theorem of Morimoto, Sakuma and Yokota \cite{MSY}.
A knot $K$ is called a tunnel number one knot
if there is an arc $\tau$ in $S^3$
which intersects $K$ only on its endpoints
and the closure of $S^3\setminus (K\cup \tau)$ is homeomorphic to a genus two handlebody.
Therefore the knot group of $K_s$ can have a presentation which has two generators and one relator.

It is well known that two groups $G$ and $G'$ are isomorphic
if there is a sequence of Tietze transformations such that
a presentation of $G$ is transformed into its of $G'$ along this sequence.
Although it is generally difficult to find such a sequence,
we can find the required sequence
by applying the procedure which appeared in the paper of Hilden, Tejada and Toro \cite{HTT} as follows.
At the first step, we obtain the Wirtinger presentation $G_1$ of the knot $K$.
Then we collapse one crossing of the knot diagram and get a graph $\Gamma$ which is thought
as a resulting object $K\cup \tau$
because the exteriors of $\Gamma$ and $K\cup \tau$ in $S^3$ are homeomorphic.
We modify $\Gamma$ with local moves in sequence forward to the shape $S^1 \vee S^1$,
and in the same time we modify the presentations by a Tietze transformation which corresponds to each local move.
In the sequel we finally obtain the graph which is homeomorphic to $S^1 \vee S^1$ and the corresponding presentation
which has two generators and one relator.

In order to apply this procedure to the case of $K_s$, we add some new local moves which are not treated in \cite{HTT},
and we refer the sequence of modifications which appeared in the paper of Kobayashi \cite{Kob}
to obtain our sequence of modifications.

In the next subsection, we will enumerate the local moves and the correspondence
between Tietze transformations and these local moves.

\subsection{The list of local moves and corresponding Tietze transformations}

Before we mention how to make the sequence of modifications of $K_s$,
we make a list of the following modifications and corresponding Tietze transformations
which will be used in our sequence.
In this list,
the corkscrew move and the open and collapsing move is originally in ~\cite{HTT},
and the others are new moves which are used in our calculation.
In the descriptions below, a formula which looks like
`` $r_1:ab=cd$ '' is a relator modified by a Tietze transformation
corresponds to a local move indicated in these figures.
These moves correspond to Tietze transformations which increase or decrease a number of generators,
but we do not indicate which generator is added or eliminated in this figures. 

\vspace{1mm}
\noindent{\small \fbox{corkscrew move}} %corkscrew move
\begin{center}
\begin{minipage}[h]{25mm}
	\centerline{\includegraphics[keepaspectratio]{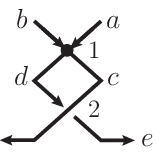}}
\end{minipage}
\begin{minipage}[h]{20mm}
\begin{tabular}{ll}
${\rm r_1:}\; ab=cd$ \\
${\rm r_2:}\; ec=cd$
\end{tabular}
\end{minipage}
\begin{minipage}[h]{20mm}
	\centerline{$\Longleftrightarrow$}
\end{minipage}
\begin{minipage}[h]{25mm}
	\centerline{\includegraphics[keepaspectratio]{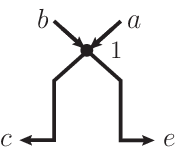}}
\end{minipage}
\begin{minipage}[h]{20mm}
\begin{tabular}{ll}
${\rm r_1:}\; ab=ec$
\end{tabular}
\end{minipage}
\end{center}

\noindent{\small \fbox{open and collapsing move}} %open collapsing move
\begin{center}
\begin{minipage}[h]{25mm}
	\centerline{\includegraphics[keepaspectratio]{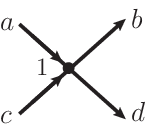}}
\end{minipage}
\begin{minipage}[h]{20mm}
\begin{tabular}{ll}
${\rm r_1:}\; bd=ac$
\end{tabular}
\end{minipage}
\begin{minipage}[h]{20mm}
	\centerline{$\Longleftrightarrow$}
\end{minipage}
\begin{minipage}[h]{25mm}
	\centerline{\includegraphics[keepaspectratio]{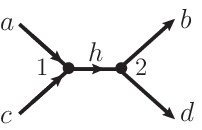}}
\end{minipage}
\begin{minipage}[h]{20mm}
\begin{tabular}{ll}
${\rm r_1:}\; h=ac$ \\
${\rm r_2:}\; bd=h$
\end{tabular}
\end{minipage}
\end{center}

\noindent{\small \fbox{upper sliding move}} %upper sliding move
\begin{center}
\begin{minipage}[h]{25mm}
	\centerline{\includegraphics[keepaspectratio]{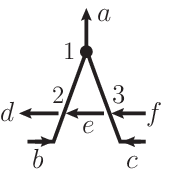}}
\end{minipage}
\begin{minipage}[h]{20mm}
\begin{tabular}{ll}
${\rm r_1:}\; a=bc$ \\
${\rm r_2:}\; db=be$ \\
${\rm r_3:}\; ec=cf$
\end{tabular}
\end{minipage}
\begin{minipage}[h]{20mm}
	\centerline{$\Longleftrightarrow$}
\end{minipage}
\begin{minipage}[h]{25mm}
	\centerline{\includegraphics[keepaspectratio]{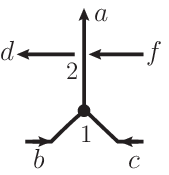}}
\end{minipage}
\begin{minipage}[h]{20mm}
\begin{tabular}{ll}
${\rm r_1:}\; a=bc$ \\
${\rm r_2:}\; da=af$
\end{tabular}
\end{minipage}
\end{center}

\noindent{\small \fbox{under sliding move}} %under sliding move
\begin{center}
\begin{minipage}[h]{25mm}
	\centerline{\includegraphics[keepaspectratio]{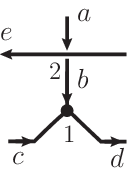}}
\end{minipage}
\begin{minipage}[h]{20mm}
\begin{tabular}{ll}
${\rm r_1:}\; d=bc$ \\
${\rm r_2:}\; be=ea$
\end{tabular}
\end{minipage}
\begin{minipage}[h]{20mm}
	\centerline{$\Longleftrightarrow$}
\end{minipage}
\begin{minipage}[h]{25mm}
	\centerline{\includegraphics[keepaspectratio]{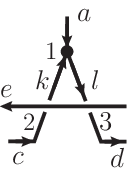}}
\end{minipage}
\begin{minipage}[h]{20mm}
\begin{tabular}{ll}
${\rm r_1:}\; l=ak$ \\
${\rm r_2:}\; ce=ek$ \\
${\rm r_3:}\; de=el$ \\
\end{tabular}
\end{minipage}
\end{center}

\noindent{\small \fbox{turning move}} %turning move
\begin{center}
\begin{minipage}[h]{25mm}
	\centerline{\includegraphics[keepaspectratio]{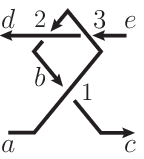}}
\end{minipage}
\begin{minipage}[h]{20mm}
\begin{tabular}{ll}
${\rm r_1:}\; ba=ac$ \\
${\rm r_2:}\; bd=da$ \\
${\rm r_3:}\; da=ae$ 
\end{tabular}
\end{minipage}
\begin{minipage}[h]{20mm}
	\centerline{$\Longleftrightarrow$}
\end{minipage}
\begin{minipage}[h]{25mm}
	\centerline{\includegraphics[keepaspectratio]{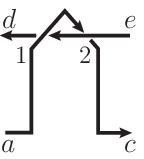}}
\end{minipage}
\begin{minipage}[h]{20mm}
\begin{tabular}{ll}
${\rm r_1:}\; da=ae$ \\
${\rm r_2:}\; ce=ea$
\end{tabular}
\end{minipage}
\end{center}

\subsection{The calculation of the presentation of the fundamental group}

In order to simplify the explanation of how to construct the sequence of modifications of $K_s$
and corresponding presentations, we divide these sequences into the several steps.

\noindent\fbox{STEP 1}
We fix an orientation of $K_s$ as in Figure~\ref{fig_G_1}. 
Then we take a Wirtinger presentation $G_1$ of $\pi_1(S^3\setminus K_s)$ as follows.

\begin{tabular}{lllllll}
$G_1=$ & $\langle$ & \multicolumn{4}{l}{$a,b,c,d,e,f,f_1,f_2,\ldots,f_{2s-2},f_{2s-1},g \;\mid\;$} & \\
 & & ${\rm r_1:}\; ca=ad$ & ${\rm r_2:}\; ac=cb$ & ${\rm r_3:}\; df=fe$ & ${\rm r_4:}\; fe=ec$ & \\
 & & ${\rm r_5:}\; ec=cg$ & ${\rm r_6:}\; f_1 a=af$ & ${\rm r_7:}\; f_2f_1=f_1a$ & ${\rm r_8:}\; f_3f_2=f_2f_1$ & \\
 & & $\phantom{\rm r_7:}\; \vdots$ & \\
 & & \multicolumn{4}{l}{${\rm r_{2s+4}:}\; f_{2s-1}f_{2s-2}=f_{2s-2}f_{2s-3}$ \;\; ${\rm r_{2s+5}:}\; gf_{2s-1}=f_{2s-1}f_{2s-2}$} & \\
 & & \multicolumn{4}{l}{${\rm r_{2s+6}:}\; bg=gf_{2s-1}$} & $\rangle$
\end{tabular}

\begin{figure}[h]
	\centerline{\includegraphics[keepaspectratio]{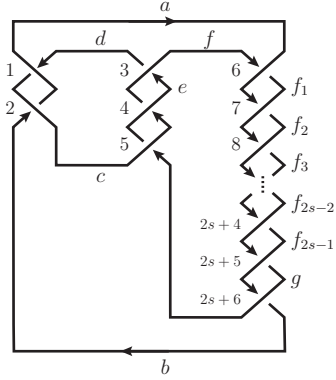}}
\caption{$K_s$ with labels of arcs and crossing points}
\label{fig_G_1}
\end{figure}

In the presentation,
each generator corresponds to the label of each arc and
each relator ${\rm r_i}$ to its of each crossing respectively.

\noindent\fbox{STEP 2}
Next we collapse the crossing labeled $6$ and obtain the graph as Figure~\ref{fig_G_gamma_G_3-2s} (a).
In the viewpoint of a complement,
$S^3\setminus \Gamma$ is homeomorphic to $S^3\setminus (K_s\cup\tau)$
where $\tau$ is the tunnel attached at the crossing labeled as $6$.
Taking this modification of $K_s$ into $\Gamma$,
we take the following presentation $G_2$
which corresponds to the result of the Tietze transformation
that we add the generator $f_0$ and the relator $f_0\bar{a}$
(we denote $a^{-1}=\bar{a}$, and we label it by ${\rm r}_{\infty}$),
and we also modify two relators ${\rm r_6}$ and ${\rm r_7}$ which contain the generator $a$.
Because of the addition of the relator $f_0\bar{a}$, we can erase the new generator $f_0$ by using this relator.
Therefore $G_2$ is isomorphic to $G_1$.

\begin{tabular}{lllllll}
$G_2=$ & $\langle$ & \multicolumn{4}{l}{$a,b,c,d,e,f,{f_0},f_1,f_2,\ldots,f_{2s-2},f_{2s-1},g \;\mid\;$} & \\
 & & ${\rm r_1:}\; ca=ad$ & ${\rm r_2:}\; ac=cb$ & ${\rm r_3:}\; df=fe$ & ${\rm r_4:}\; fe=ec$ & \\
 & & ${\rm r_5:}\; ec=cg$ & ${\rm r_6:}\; f_1  {f_0}=af$ & ${\rm r_7:}\; f_2f_1=f_1  {f_0}$ & ${\rm r_8:}\; f_3f_2=f_2f_1$ & \\
 & & $\phantom{\rm r_7:}\; \vdots$ \\
 & & \multicolumn{4}{l}{${\rm r_{2s+4}:}\; f_{2s-1}f_{2s-2}=f_{2s-2}f_{2s-3}$ \;\; ${\rm r_{2s+5}:}\; gf_{2s-1}=f_{2s-1}f_{2s-2}$} & \\
 & & \multicolumn{4}{l}{${\rm r_{2s+6}:}\; bg=gf_{2s-1}$ \;\; ${\rm r_{\infty}:}\;  f_0\bar{a}$} & $\rangle$ 
\end{tabular}

\noindent\fbox{STEP 3}
We make a corkscrew move to the pair of crossings $6$ and $7$,
and iterate this operation to successive crossings.
Then we obtain Figure~\ref{fig_G_gamma_G_3-2s} (b) and the corresponding presentation $G_{3,2s}$ as follows:

\begin{tabular}{lllllll}
$G_{3,2s}=$ & $\langle$ & \multicolumn{4}{l}{$a,b,c,d,e,f,g \,\mid\;$} & \\
 & & ${\rm r_1:}\; ca=ad$ & ${\rm r_2:}\; ac=cb$ & ${\rm r_3:}\; df=fe$ & ${\rm r_4:}\; fe=ec$ & \\
 & & ${\rm r_5:}\; ec=cg$ & ${\rm r_6:}\;  {bg}=af$ & \multicolumn{2}{l}{${\rm r_{\infty}:}\;  {(\bar{g}\bar{b})^s g(bg)^s \bar{a}}$} & $\rangle$
\end{tabular}

\begin{figure}[h]
	\centerline{\includegraphics[keepaspectratio]{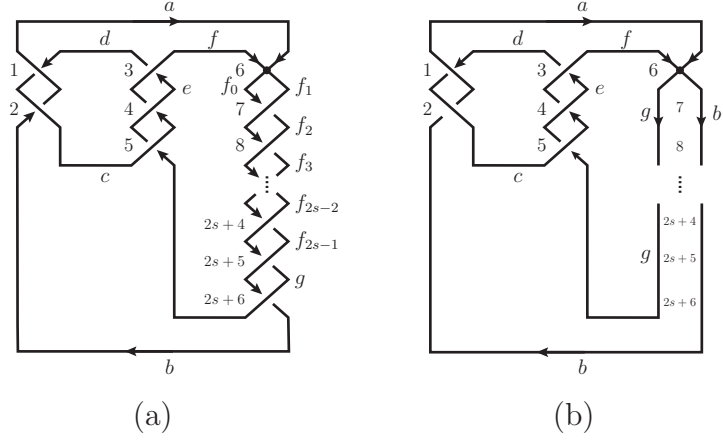}}
\caption{(a) $\Gamma$ with one tunnel and (b) the result of iterations of a corkscrew move}
\label{fig_G_gamma_G_3-2s}
\end{figure}

We will explain these operations precisely as follows.
As we mentioned above, we first make a corkscrew move to the pair $6$ and $7$.
This operation causes that every generator $f_0$ which appeared in the relators is replaced by $\bar{f}_1 f_2 f_1$
obtained from the relator ${\rm r_7:}\; f_2 f_1=f_1 f_0$.
Then $f_0$ is erased in the generators and all relators.
The resultant presentation is denoted by $G_{3,1}$ and the relator $r_{\infty}$ by $R_1:\bar{f}_1f_2f_1\bar{a}$
temporarily in this step.

Next we make a corkscrew move to the pair $6$ and $8$ successively,
then we get the presentation $G_{3,2}$ which loses one generator $f_1$ by using the relator ${\rm r_8:}\;f_3f_2=f_2f_1 \Leftrightarrow f_1=\bar{f}_2 f_3 f_2$, and we have $R_2:\bar{f}_2 \bar{f}_3 f_2 f_3 f_2\bar{a}$.

By repeating this operation,
we get the sequence of presentations $G_{3,3}$, $G_{3,4}$, $G_{3,5}$, $\cdots$ which have the specific relators
\begin{align*}
R_3&:\bar{f}_3\bar{f}_4\bar{f}_3 f_4 f_3 f_4 f_3\bar{a}, \\
R_4&:\bar{f}_4\bar{f}_5\bar{f}_4\bar{f}_5 f_4 f_5 f_4 f_5 f_4 \bar{a}, \\
R_5&:\bar{f}_5\bar{f}_6\bar{f}_5\bar{f}_6\bar{f}_5 f_6 f_5 f_6 f_5 f_6 f_5 \bar{a}, \\
&\cdots .
\end{align*}

By this observation,
we can assume that the specific relator $R_i$ in the presentation $G_{3,i}$ has the following formulae.

When $i$ is odd, let $j=1,2,3,\cdots$, and we set
\begin{align*}
R_i=R_{2j-1}&=
\left(\prod_{n=1}^{j-1} \left(\bar{f}_{2j-1}\bar{f}_{2j}\right)\right)\bar{f}_{2j-1}\left(\prod_{n=1}^{j}\left(f_{2j}f_{2j-1}\right)\right)\bar{a}\, ,
\end{align*}
and in this case if $j=1$ we assume that the part $\prod_{n=1}^{j-1} \left(\bar{f}_{2j-1}\bar{f}_{2j}\right)$ is equal to identity.

When $i$ is even, let $j=1,2,\cdots$, and we set
\begin{align*}
R_i=R_{2j}&=
\left(\prod_{n=1}^{j}\left(\bar{f}_{2j}\bar{f}_{2j+1}\right)\right)f_{2j}\left(\prod_{n=1}^{j}\left(f_{2j+1}f_{2j}\right)\right)\bar{a}\, .
\end{align*}

In order to verify this formulae, we take induction as follows.
We assume that the formula $R_i$ is correct when $i=2k-1$ is odd.
Then we see the relator ${\rm r}_{i+7}: f_{i+2}f_{i+1}=f_{i+1}f_i \Leftrightarrow f_i=\bar{f}_{i+1}f_{i+2}f_{i+1}$
in the presentation $G_{3,i}$,
and it follows $f_{2k-1}=\bar{f}_{2k}f_{2k+1}f_{2k}$.
By making a corkscrew move to the pair $6$ and $6+(i+1)$, we replace these $f_{2k-1}$ and $\bar{f}_{2k-1}$ which
appeared in the relations of $G_{3,i}$.
In particular,
\begin{align*}
R_i&=R_{2k-1} \\
&=\left(\prod_{n=1}^{k-1} \left(\bar{f}_{2k-1}\bar{f}_{2k}\right)\right)\bar{f}_{2k-1}\left(\prod_{n=1}^{k}\left(f_{2k}f_{2k-1}\right)\right)\bar{a} \\
&\rightarrow \left(\prod_{n=1}^{k-1}\left(\left(\bar{f}_{2k}\bar{f}_{2k+1}f_{2k}\right)\bar{f}_{2k}\right)\right)\left(\bar{f}_{2k}\bar{f}_{2k+1}f_{2k}\right)\left(\prod_{n=1}^{k}\left(f_{2k}\left(\bar{f}_{2k}f_{2k+1}f_{2k}\right)\right)\right)\bar{a} \\
&=\left(\prod_{n=1}^{k-1}\left(\bar{f}_{2k}\bar{f}_{2k+1}\right)\right)\bar{f}_{2k}\bar{f}_{2k+1}f_{2k}\left(\prod_{n=1}^{k}\left(f_{2k+1}f_{2k}\right)\right)\bar{a} \\
&=\left(\prod_{n=1}^{k}\left(\bar{f}_{2k}\bar{f}_{2k+1}\right)\right)f_{2k}\left(\prod_{n=1}^{k}\left(f_{2k+1}f_{2k}\right)\right)\bar{a} \\
&=R_{2k}=R_{i+1}.
\end{align*}

For the case when $i$ is even, we can also verify the formula of $R_i=R_{2k}$ similarly.

Next we have to observe the last three steps $G_{3,2s-2}$, $G_{3,2s-1}$ and $G_{3,2s}$.
By the above formula,
$G_{3,2s-2}$ has the specific relator
$$R_{2s-2}=
\left(\prod_{n=1}^{s-1}\left(\bar{f}_{2s-2}\bar{f}_{2s-1}\right)\right)f_{2s-2}\left(\prod_{n=1}^{s-1}\left(f_{2s-1}f_{2s-2}\right)\right)\bar{a}.$$
We make a corkscrew move to the pair $6$ and $2s+5$
using the relator ${\rm r}_{2s+5}:g f_{2s-1}=f_{2s-1}f_{2s-2}$,
which is equivalent to $f_{2s-2}=\bar{f}_{2s-1}g f_{2s-1}$ and $\bar{f}_{2s-2}=\bar{f}_{2s-1}\bar{g}f_{2s-1}$.
Then we obtain the presentation $G_{3,2s-1}$ with the specific relator
$$R_{2s-1}=
\left(\prod_{n=1}^{s-1}\left(\bar{f}_{2s-1}\bar{g}\right)\right)\bar{f}_{2s-1}\left(\prod_{n=1}^{s}\left(g f_{2s-1}\right)\right)\bar{a}.$$

At last we make a corkscrew move to $6$ and $2s+6$ using the relator ${\rm r}_{2s+6}:bg=gf_{2s-1}$
which is equivalent to $f_{2s-1}=\bar{g}bg$ and $\bar{f}_{2s-1}=\bar{g}\bar{b}g$.
Then we obtain the presentation $G_{3,2s}$ with the specific relator $R_{2s}$ as follows:
\begin{align*}
R_{2s}
&=\left(\prod_{n=1}^{s-1}\left(\left(\bar{g}\bar{b}g\right)\bar{g}\right)\right)\left(\bar{g}\bar{b}g\right)\left(\prod_{n=1}^{s}\left(g\left(\bar{g}bg\right)\right)\right)\bar{a} \\
&=\left(\prod_{n=1}^{s-1}\left(\bar{g}\bar{b}\right)\right)\left(\bar{g}\bar{b}g\right)\left(\prod_{n=1}^{s}\left(bg\right)\right)\bar{a} \\
&=\left(\prod_{n=1}^{s}\left(\bar{g}\bar{b}\right)\right)g\left(\prod_{n=1}^{s}\left(bg\right)\right)\bar{a} \\
&=\left(\bar{g}\bar{b}\right)^s g\left(bg\right)^s \bar{a}.
\end{align*}

\noindent\fbox{STEP 4}
Before starting the next step,
we remark that the graph in Figure~\ref{fig_G_gamma_G_3-2s} (b) is equivalent to Figure~\ref{fig_G_3-2s_equiv}.

\begin{figure}[h]
	\centerline{\includegraphics[keepaspectratio]{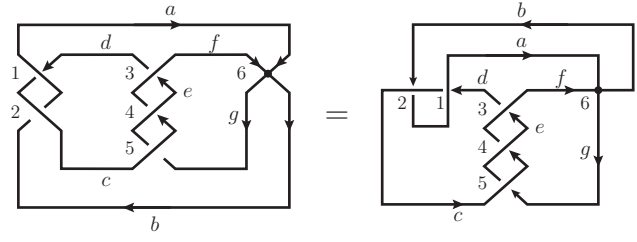}}
\caption{the graph before step 4}
\label{fig_G_3-2s_equiv}
\end{figure}

We make a opening move to the point $6$ and obtain the graph in Figure~\ref{fig_G_4_5_6} (a).
This operation corresponds to Tietze transformations so that
one generator $h$ is added to the generators and
one new relation $h=af$ is added to the relations.
And we replace the element $af$ which appeared in the relator ${\rm r_6}$ by $h$,
replace $\bar{g}\bar{b}$ by $\bar{h}$
and $bg$ by $h$ in the relator ${\rm r_{\infty}}$.
Then we obtain the presentation $G_4$ as follows:

\begin{tabular}{lllllll}
$G_{4}=$ & $\langle$ & \multicolumn{4}{l}{$a,b,c,d,e,f,g,h \,\mid\;$} & \\
 & & ${\rm r_1:}\; ca=ad$ & ${\rm r_2:}\; ac=cb$ & ${\rm r_3:}\; df=fe$ & ${\rm r_4:}\; fe=ec$ & \\
 & & ${\rm r_5:}\; ec=cg$ & ${\rm r_6:}\;  {bg}=h$ & ${\rm r_7:}\;  h=af$ & ${\rm r_{\infty}:}\;  \bar{h}^s gh^s \bar{a}$ & $\rangle$.
\end{tabular}

\noindent\fbox{STEP 5}
Next we apply an upper sliding move to the triplet $7$, $3$ and $1$,
and obtain the graph in Figure~\ref{fig_G_4_5_6} (b).
In this modification,
the element $d$ is replaced by $\bar{a}ca$ obtained from the relation ${\rm r_1:}\; ca=ad$,
and $d$ is eliminated from the generators.
Then we obtain the presentation

\begin{tabular}{lllllll}
$G_{5}=$ & $\langle$ & \multicolumn{4}{l}{$a,b,c,e,f,g,h \,\mid\;$} & \\
 & & ${\rm r_2:}\; ac=cb$ & ${\rm r_3:}\; ch=he$ & ${\rm r_4:}\; fe=ec$ & ${\rm r_5:}\; ec=cg$ & \\
 & & ${\rm r_6:}\; {bg}=h$ & ${\rm r_7:}\; h=af$ & \multicolumn{2}{l}{${\rm r_{\infty}:}\;  \bar{h}^s gh^s \bar{a}$} & $\rangle$.
\end{tabular}

\noindent\fbox{STEP 6}
We make an under sliding move to the pair $7$ and $2$ and obtain the graph in Figure~\ref{fig_G_4_5_6} (c),
then the generator $a$ is eliminated and two generators $k$ and $l$ are added.
This operation is divided into the following two steps.
In the first step, by using the relator ${\rm r}_7: h=af$ $\Leftrightarrow$ $a=h\bar{f}$,
the generator $a$ which appeared in the relations ${\rm r}_2$ and ${\rm r}_{\infty}$ is replaced by $h\bar{f}$.
Next we add two generators $k$ and $l$, and two relations ${\rm r_8}:\;fc=ck$ and ${\rm r_9:}\;hc=cl$.
Then we obtain the presentation

\begin{tabular}{lllllll}
${G_6}'=$ & $\langle$ & \multicolumn{4}{l}{$b,c,e,f,g,h,k,l \,\mid\;$} & \\
 & & ${\rm r_2:}\; h\bar{f}c=cb$ & ${\rm r_3:}\; ch=he$ & ${\rm r_4:}\; fe=ec$ & ${\rm r_5:}\; ec=cg$ & \\
 & & ${\rm r_6:}\; {bg}=h$ & ${\rm r_8:}\;fc=ck$ & ${\rm r_9:}\;hc=cl$ & ${\rm r_{\infty}:}\;  \bar{h}^s gh^s (f\bar{h})$ & $\rangle$.
\end{tabular}

By using the relations ${\rm r_8:}\;fc=ck$ $\Leftrightarrow$ $c=\bar{f}ck$ and $c\bar{k}=\bar{f}c$,
and ${\rm r_9:}\; hc=cl$,
we rewrite the relation ${\rm r_2:}\; h\bar{f}c=cb$ as follows:
\begin{align*}
h\bar{f}c=cb &\Leftrightarrow h(c\bar{k})=cb \\
&\Leftrightarrow (cl)\bar{k}=cb \\
&\Leftrightarrow l\bar{k}=b \Leftrightarrow l=bk,
\end{align*}
and we rename the number of this relation by $10$.
Consequently, we obtain the presentation

\begin{tabular}{lllllll}
${G_6}=$ & $\langle$ & \multicolumn{4}{l}{$b,c,e,f,g,h,k,l \,\mid\;$} & \\
 & & ${\rm r_{10}:}\; l=bk$ & ${\rm r_3:}\; ch=he$ & ${\rm r_4:}\; fe=ec$ & ${\rm r_5:}\; ec=cg$ & \\
 & & ${\rm r_6:}\; {bg}=h$ & ${\rm r_8:}\;fc=ck$ & ${\rm r_9:}\;hc=cl$ & ${\rm r_{\infty}:}\;  \bar{h}^s gh^s f\bar{h}$ & $\rangle$.
\end{tabular}

\begin{figure}[h]
	\centerline{\includegraphics[keepaspectratio]{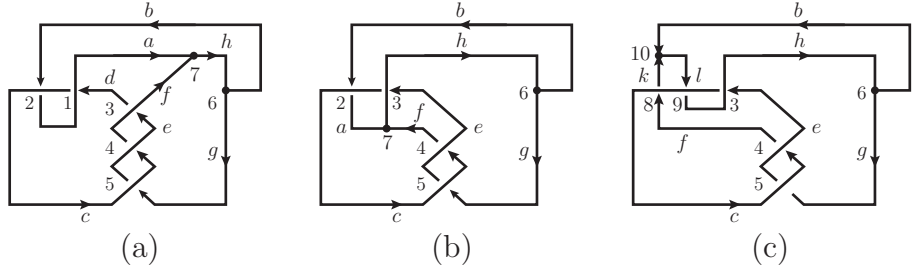}}
\caption{the graph of step 4, 5 and 6}
\label{fig_G_4_5_6}
\end{figure}

\noindent\fbox{STEP 7}
We carry out a collapsing move of the arc $b$ and we get the graph in Figure~\ref{fig_G_7_8} (a).
Notice that the two graphs in Figure~\ref{fig_G_7_8} (a) are equivalent.
By this operation, the generator $b$ is eliminated by using the relation ${\rm r_{10}:}\;l=bk$.
Then we obtain the presentation

\begin{tabular}{lllllll}
${G_7}=$ & $\langle$ & \multicolumn{4}{l}{$c,e,f,g,h,k,l \,\mid\;$} & \\
 & & ${\rm r_3:}\; ch=he$ & ${\rm r_4:}\; fe=ec$ & ${\rm r_5:}\; ec=cg$ & ${\rm r_6:}\; h\bar{g}=l\bar{k}$ & \\
 & & ${\rm r_8:}\;fc=ck$ & ${\rm r_9:}\;hc=cl$ & \multicolumn{2}{l}{${\rm r_{\infty}:}\;  \bar{h}^s gh^s f\bar{h}$} & $\rangle$.
\end{tabular}

\noindent\fbox{STEP 8}
We apply a turning move to the triplet of crossings $4$, $5$ and $8$,
and we get the graph in Figure~\ref{fig_G_7_8} (b).
This operation is divided into two steps.
In the beginning,
the generator $f$ is eliminated by using the relation ${\rm r_4:}\;fe=ec$ $\Leftrightarrow$ $f=ec\bar{e}$,
and then we obtain the presentation

\begin{tabular}{lllllll}
${G_8}'=$ & $\langle$ & $c,e,g,h,k,l \,\mid\;$ & ${\rm r_3:}\; ch=he$ & ${\rm r_5:}\; ec=cg$ & ${\rm r_6:}\; h\bar{g}=l\bar{k}$ & \\
& & & ${\rm r_8:}\;(ec\bar{e})c=ck$ & ${\rm r_9:}\;hc=cl$ & ${\rm r_{\infty}:}\;  \bar{h}^s gh^s (ec\bar{e})\bar{h}$ & $\rangle$.
\end{tabular}

Next by using the relation ${\rm r_5:}\; ec=cg$
we replace the generator $e$ which appeared in the relation ${\rm r_8}$ by $e=cg\bar{c}$ and $\bar{e}=c\bar{g}\bar{c}$,
then we obtain the relation $gc=kg$ and we rename the number of it by $11$.
As a result, we obtain the presentation

\begin{tabular}{lllllll}
${G_8}=$ & $\langle$ & $c,e,g,h,k,l \,\mid\;$ & ${\rm r_3:}\; ch=he$ & ${\rm r_5:}\; ec=cg$ & ${\rm r_6:}\; h\bar{g}=l\bar{k}$ & \\
& & & ${\rm r_{11}:}\;gc=kg$ & ${\rm r_9:}\;hc=cl$ & ${\rm r_{\infty}:}\;  \bar{h}^s gh^s ec\bar{e}\bar{h}$ & $\rangle$.
\end{tabular}

\begin{figure}[h]
	\centerline{\includegraphics[keepaspectratio]{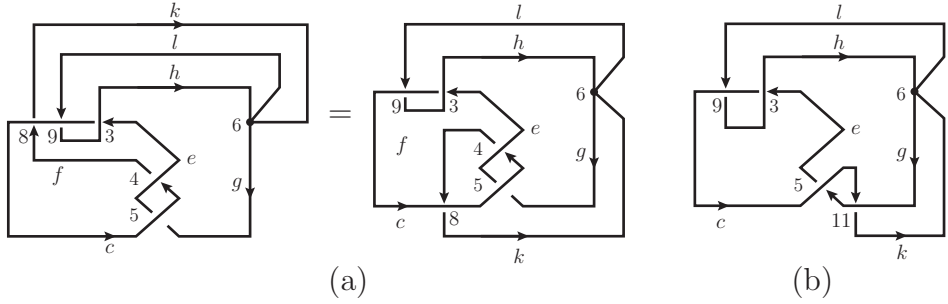}}
\caption{the graph of step 7 and 8}
\label{fig_G_7_8}
\end{figure}

\noindent\fbox{STEP 9}
We make a corkscrew move to the pair $6$ and $11$, then we obtain the graph in Figure~\ref{fig_G_9to12} (a).
In this operation the generator $k$ is eliminated by using the relation ${\rm r_{11}:}\;gc=kg$, then we obtain

\begin{tabular}{lllllll}
${G_9}=$ & $\langle$ & $c,e,g,h,l \,\mid\;$ & ${\rm r_3:}\; ch=he$ & ${\rm r_5:}\; ec=cg$ & ${\rm r_6:}\; hc=lg$ & \\
& & & ${\rm r_9:}\;hc=cl$ & \multicolumn{2}{l}{${\rm r_{\infty}:}\;  \bar{h}^s gh^s ec\bar{e}\bar{h}$} & $\rangle$.
\end{tabular}

\noindent\fbox{STEP 10}
By a corkscrew move of the pair $6$ and $5$, we obtain the graph in Figure~\ref{fig_G_9to12} (b),
the generator $g$ is eliminated by the relation ${\rm r_5:}\;ec=cg$,
then we obtain

\begin{tabular}{lllllll}
${G_{10}}=$ & $\langle$ & $c,e,h,l \,\mid\;$ & ${\rm r_3:}\; ch=he$ & ${\rm r_6:}\; h\bar{e}=l\bar{c}$ & ${\rm r_9:}\;hc=cl$ & \\
& & & \multicolumn{3}{l}{${\rm r_{\infty}:}\;  \bar{h}^s (\bar{c}ec) h^s ec\bar{e}\bar{h}$} & $\rangle$.
\end{tabular}

\noindent\fbox{STEP 11}
By a corkscrew move to the pair $6$ and $3$, we obtain the graph in Figure~\ref{fig_G_9to12} (c),
the generator $e$ is eliminated by the relation ${\rm r_3:}\;ch=he$,
then we obtain

\begin{tabular}{llllll}
${G_{11}}=$ & $\langle$ & $c,h,l \,\mid\;$ & ${\rm r_6:}\; hc=cl$ & ${\rm r_9:}\;hc=cl$ & \\
& & & \multicolumn{2}{l}{${\rm r_{\infty}:}\;  \bar{h}^s \bar{c}\bar{h}chc h^{s-1} chc\bar{h}\bar{c}$} & $\rangle$.
\end{tabular}

\noindent\fbox{STEP 12}
Finally, we make a corkscrew move to the pair $6$ and $9$, we obtain the last graph in Figure~\ref{fig_G_9to12} (d),
the generator $h$ is eliminated by the relation ${\rm r_6}={\rm r_9}$.
We obtain the following presentation:

\begin{tabular}{lllll}
${G_{12}}=$ & $\langle$ & $c,l \,\mid\;$ 
& ${\rm r_{\infty}:}\; c\bar{l}^s \bar{c}\bar{l} clc l^{s-1} clc \bar{l} \bar{c}^2$
& $\rangle$.
\end{tabular}

In this presentation, we simplify ${\rm r_{\infty}}$ by modifying cyclically,
then we obtain the following final version of the presentation and it is denoted by $G_{K_s}$.

\begin{tabular}{lllll}
${G_{K_s}}=$ & $\langle$ & $c,l \,\mid\;$ 
& $clc \bar{l}\bar{c} \bar{l}^s \bar{c}\bar{l} clc l^{s-1}$
& $\rangle$.
\end{tabular}

Hilden, Tejada and Toro proved that
for a tunnel number one knot $K$,
the fundamental group $\pi_1(S^3\setminus K)$ has a two generator, one relator presentation
in which the relator is a palindrome~\cite[Theorem 5.3]{HTT}.
A word is called a palindrome if its spelling can be read backwards as same as forwards.
We can make sure that the presentation $G_{K_s}$ has two generators and one relator,
especially the relator is a palindrome.

\begin{figure}[h]
	\centerline{\includegraphics[keepaspectratio]{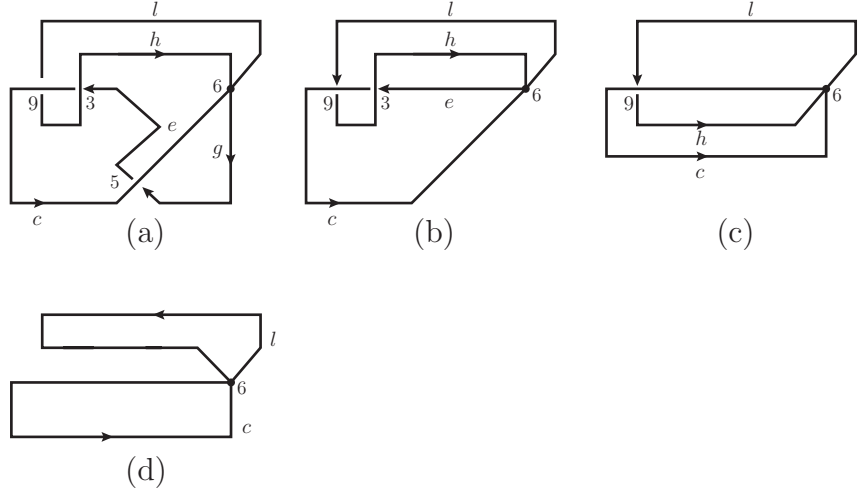}}
\caption{the graph of step from 9 to 12}
\label{fig_G_9to12}
\end{figure}

\subsection{The calculation of the presentation of a meridian-longitude pair}

Let $E_{K_s}(p/q)$ be a closed manifold obtained by Dehn surgery along $K_s$ with a slope $p/q$.
In order to obtain a presentation of $G_{K_s}(p,q)=\pi_1(E_{K_s}(p/q))$,
we have to get a presentation of a meridian-longitude pair.
In this subsection we will calculate it which is compatible with the final presentation $G_{K_s}$
obtained in the previous subsection.
The way of the calculation is as follows.
We first fix a meridian $c$ and get a presentation of a longitude $L_1$ which are compatible with the first presentation $G_1$ by using the method which appeared in the book of Burde and Zieschang~\cite{BZ}.
Then we continue to modify $L_i$ from $i=1$ to the last presentation $L_{12}$ compatible with $G_{12}$
along the steps mentioned in the previous section.

So we are going to begin the calculation.
We fix the meridian which is presented by the generator $c$.
The initial presentation $L_1$ of the longitude is obtained by the following procedure.
We read the label of arcs starting on the arc $c$ forward to the opposite direction of the orientation of $K_s$.
The sign of each label is determined so that
when we pass through the arc $x_k$ from the bottom,
if the orientation of the over arc $x_k$ coincides with its of the under arcs when we rotate $x_k$ counterclockwise,
we assign the positive sign,
otherwise we assign the negative sign.
When we have just come back to the initial arc $c$, we add the power of generator $c^\alpha$ so that
$\alpha$ is inverse of the sum of all signs we have read.
Then we obtain
$$L_1=
afc f_{2s-1} f_{2s-3} \cdots f_1 cg f_{2s-2} \cdots f_2 ae (\bar{c})^{2s+6} .
$$

Next we will modify $L_1$ along the steps of the modification of $G_i$.
In the STEP 1, 2 and the first step of STEP 3, the $L_1$ is not affected and we set $L_1=L_{3,1}$.
In the second step of STEP 3,
by using the relation ${\rm r_8:}\;f_3f_2=f_2f_1$
we eliminate the generator $f_1$ appeared in $L_{3,1}$ and we obtain
$$L_{3,2}=
afc f_{2s-1} f_{2s-3} \cdots f_3(\bar{f}_2 f_3 f_2) cg f_{2s-2} \cdots f_2 ae (\bar{c})^{2s+6} .
$$

By repeating this operation, we obtain the sequence of $L_{3,i}$, $i=3, 4, 5, 6, \cdots$,
at each step using the relator ${\rm r_{i+6}:}\;f_{i+1} f_i=f_i f_{i-1}$, as follows:
\begin{align*}
L_{3,3}&=afc f_{2s-1} \cdots f_5 \bar{f}_4 f_3 f_4 f_3 cg f_{2s-2} \cdots f_4 \bar{f}_3 f_4 f_3 ae (\bar{c})^{2s+6} \\
L_{3,4}&=afc f_{2s-1} \cdots f_5 (\bar{f}_4)^2 f_5 f_4 f_5 f_4 cg f_{2s-2} \cdots f_6 \bar{f}_5 f_4 f_5 f_4 ae (\bar{c})^{2s+6} \\
L_{3,5}&=afc f_{2s-1} \cdots f_7 (\bar{f}_6)^2 f_5 f_6 f_5 f_6 f_5 cg f_{2s-2} \cdots f_6 (\bar{f}_5)^2 f_6 f_5 f_6 f_5 ae (\bar{c})^{2s+6} \\
L_{3,6}&=afc f_{2s-1} \cdots f_7 (\bar{f}_6)^3 f_7 f_6 f_7 f_6 f_7 f_6 cg f_{2s-2} \cdots f_8 (\bar{f}_7)^2 f_6 f_7 f_6 f_7 f_6 ae (\bar{c})^{2s+6} \\
& \cdots.
\end{align*}

In order to simplify the description of $L_{3,i}$ we take the temporary notation that
$L_{3,i}=afc(\mathfrak{L}_i) cg (\mathfrak{R}_i) ae (\bar{c})^{2s+6}$. 
By the above observation we can assume that $\mathfrak{L}_i$ and $\mathfrak{R}_i$ have the formulae as follows.
We divided the case if $i=2j-1$ is odd or $i=2j$ is even ($j=1,2,3,\cdots, s$).
$\mathfrak{L}_i$ has the following formula
\begin{align*}
\mathfrak{L}_{2j-1}&=
\left(\prod_{n=s}^{j+1} f_{2n-1}\right) (\bar{f}_{2j})^{j-1} f_{2j-1} \left( f_{2j} f_{2j-1} \right)^{j-1} \\
\mathfrak{L}_{2j}&=
\left(\prod_{n=s}^{j+1} f_{2n-1}\right) (\bar{f}_{2j})^{j} \left( f_{2j+1} f_{2j} \right)^j ,
\end{align*}
where if $j=1$ we assume $(\bar{f}_{2j})^{j-1}$ and $\left( f_{2j} f_{2j-1} \right)^{j-1}$ are equal to identity.
$\mathfrak{R}_i$ has the following formula
\begin{align*}
\mathfrak{R}_{2j-1}&=
\left(\prod_{n=s-1}^j f_{2n} \right) (\bar{f}_{2j-1})^{j-1} \left( f_{2j} f_{2j-1} \right)^{j-1} \\
\mathfrak{R}_{2j}&=
\left(\prod_{n=s-1}^{j+1} f_{2n}\right) (\bar{f}_{2j+1})^{j+1} f_{2j} \left( f_{2j+1} f_{2j} \right)^{j-1} ,
\end{align*}
where if $j=1$ we assume
$(\bar{f}_{2j-1})^{j-1}$ and $\left( f_{2j} f_{2j-1} \right)^{j-1}$ are equal to identity,
and also $(\bar{f}_{2j+1})^{j-1}$ and $\left( f_{2j+1} f_{2j} \right)^{j-1}$ are equal to identity.
Notice that in these formulae we temporarily use the symbol $\prod_{n=a}^{b} x_n$ in the unusual meaning so that
the index of $x_n$ is decreasing in the product
as $\prod_{n=a}^{b} x_n=x_a x_{a-1} x_{a-2} \cdots x_{b+2} x_{b+1} x_b$.

We can verify it by induction as follows.
The above formulae are correct when $(i,j)=(1,1)$, $(2,1)$, $(3,2)$ and $(4,2)$.
If we assume that the formulae are correct when $i=2k-1$,
we obtain
\begin{align*}
\mathfrak{L}_i=\mathfrak{L}_{2k-1}&=
\left(\prod_{n=s-1}^{k+1} f_{2n-1}\right) (\bar{f}_{2k})^{k-1} f_{2k-1} \left( f_{2k} f_{2k-1} \right)^{k-1} & {\rm (*)} \\
\mathfrak{R}_i=\mathfrak{R}_{2k-1}&=
\left(\prod_{n=s-1}^{k} f_{2n}\right) (\bar{f}_{2k-1})^{k-1} \left( f_{2k} f_{2k-1} \right)^{k-1}. & {\rm (**)}
\end{align*}
Then
by using the relator ${\rm r_{i+7}:} f_{i+2} f_{i+1} =f_{i+1} f_i$ $\Leftrightarrow$ $f_{2k-1}=\bar{f}_{2k} f_{2k+1} f_{2k}$
in the presentation $G_{3,i}$,
we replace the generator $f_i=f_{2k-1}$ which appeared in $\mathfrak{L}_i$ and $\mathfrak{R}_i$ by $\bar{f}_{2k} f_{2k+1} f_{2k}$,
and then we obtain
{\allowdisplaybreaks %
\begin{align*}
{\rm (*):}& \left(\prod_{n=s}^{k+1} f_{2n-1}\right) (\bar{f}_{2k})^{k-1} \left( \bar{f}_{2k} f_{2k+1} f_{2k} \right) \left( f_{2k} \left( \bar{f}_{2k} f_{2k+1} f_{2k} \right)\right)^{k-1} \\
&=\left(\prod_{n=s}^{k+1} f_{2n-1} \right) (\bar{f}_{2k})^k \left( f_{2k+1} f_{2k} \right)^k \\
&=\mathfrak{L}_{2k}=\mathfrak{L}_{i+1}, \\
{\rm (**):}& \left(\prod_{n=s-1}^{k} f_{2n}\right) \left( \bar{f}_{2k} \bar{f}_{2k+1} f_{2k} \right)^{k-1} \left( f_{2k} \left( \bar{f}_{2k} f_{2k+1} f_{2k} \right)\right)^{k+1} \\
&=\left(\prod_{n=s-1}^{k+1} f_{2n}\right) f_{2k} \left( \bar{f}_{2k} \left( \bar{f}_{2k+1} \right)^{k-1} f_{2k} \right) \left( f_{2k+1} f_{2k} \right)^{k-1} \\
&=\left(\prod_{n=s-1}^{k+1} f_{2n} \right) \left( \bar{f}_{2k+1} \right)^{k-1} f_{2k} \left( f_{2k+1} f_{2k} \right)^{k-1} \\
&=\mathfrak{R}_{2k}=\mathfrak{R}_{i+1}.
\end{align*}
}

And also if these formulae are correct when $i=2k$, we obtain
\begin{align*}
\mathfrak{L}_i=\mathfrak{L}_{2k}&=
\left(\prod_{n=s}^{k+1} f_{2n-1} \right) (\bar{f}_{2k})^k \left( f_{2k+1} f_{2k} \right)^k & {\rm (*)} \\
\mathfrak{R}_i=\mathfrak{R}_{2k}&=
\left( \prod_{n=s-1}^{k+1} f_{2n} \right) (\bar{f}_{2k+1})^{k-1} f_{2k} \left( f_{2k+1} f_{2k} \right)^{k-1}. & {\rm (**)}
\end{align*}
Then we eliminate the generator $f_i=f_{2k}$ by the relator $f_{2k+2}f_{2k+1}=f_{2k+1}f_{2k}$ in $G_{3,i}$ and obtain
{\allowdisplaybreaks %
\begin{align*}
{\rm (*):}& \left(\prod_{n=s}^{k+1} f_{2n-1} \right) \left(\bar{f}_{2k+1} \bar{f}_{2k+2} f_{2k+1} \right)^k \left( f_{2k+1} \left( \bar{f}_{2k+1} f_{2k+2} f_{2k+1} \right) \right)^k \\
&=\left( \prod_{n=s}^{k+2} f_{2n-1} \right) f_{2k+1} \left( \bar{f}_{2k+1} (\bar{f}_{2k+2})^k f_{2k+1} \right) \left( f_{2k+2} f_{2k+1} \right)^k \\
&=\left( \prod_{n=s}^{k+2} f_{2n-1} \right) (\bar{f}_{2k+2})^k f_{2k+1} \left( f_{2k+2} f_{2k+1} \right)^k \\
&=\mathfrak{L}_{2(k+1)-1}=\mathfrak{L}_{2k+1}=\mathfrak{L}_{i+1}, \\
{\rm (**):}& \left( \prod_{n=s-1}^{k+1} f_{2n} \right) (\bar{f}_{2k+1})^{k-1} \left( \bar{f}_{2k+1} f_{2k+2} f_{2k+1} \right) \left( f_{2k+1} \left( \bar{f}_{2k+1} f_{2k+2} f_{2k+1} \right) \right)^{k-1} \\
&=\left( \prod_{n=s-1}^{k+1} f_{2n} \right) (\bar{f}_{2k+1})^k f_{2k+2} f_{2k+1} \left( f_{2k+2} f_{2k+1} \right)^{k-1} \\
&=\left( \prod_{n=s-1}^{k+1} f_{2n} \right) (\bar{f}_{2k+1})^k \left( f_{2k+2} f_{2k+1} \right)^{k} \\
&=\mathfrak{R}_{2(k+1)-1}=\mathfrak{R}_{2k+1}=\mathfrak{R}_{i+1}.
\end{align*}
}

Next we observe the last three steps of presentations $G_{3,2s-2}$, $G_{3,2s-1}$ and $G_{3,2s}$.
In the presentation $G_{3,2s-3}$ we can see
\begin{align*}
L_{3,2s-3}&=afc\mathfrak{L}_{2(s-1)-1} cg \mathfrak{R}_{2(s-1)-1} ae (\bar{c})^{2s+6} \\
&=afc f_{2s-1}(\bar{f}_{2s-2})^{s-2} f_{2s-3} \left( f_{2s-2} f_{2s-3} \right)^{s-2} \\
&\phantom{=afc} cg f_{2s-2}(\bar{f}_{2s-3})^{s-2} \left( f_{2s-2} f_{2s-3} \right)^{s-2} ae (\bar{c})^{2s+6}.
\end{align*}

We erase the generator $f_{2s-3}$ by the relation $f_{2s-1}f_{2s-2}=f_{2s-2}f_{2s-3}$, and then
\begin{align*}
L_{3,2s-2}
&=afc f_{2s-1} (\bar{f}_{2s-2})^{s-2} \left(\bar{f}_{2s-2} f_{2s-1} f_{2s-2} \right) \left( f_{2s-2} \left( \bar{f}_{2s-2} f_{2s-1} f_{2s-2} \right)\right)^{s-2} \\
&\phantom{=afc} cg f_{2s-2} \left( \bar{f}_{2s-2} \bar{f}_{2s-1} f_{2s-2} \right)^{s-2} \left( f_{2s-2} \left( \bar{f}_{2s-2} \bar{f}_{2s-1} f_{2s-2} \right)\right)^{s-2} ae (\bar{c})^{2s+6} \\
&=afc f_{2s-1} (\bar{f}_{2s-2})^{s-1} \left( f_{2s-1} f_{2s-2} \right)^{s-1} \\
&\phantom{=afc} cg (\bar{f}_{2s-1})^{s-2} f_{2s-2} \left( f_{2s-1} f_{2s-2} \right)^{s-2} ae (\bar{c})^{2s+6}.
\end{align*}

Next we erase the generator $f_{2s-2}$ by the relation $gf_{2s-1}=f_{2s-1}f_{2s-2}$ similarly, then
\begin{align*}
L_{3,2s-1}
&= afc (\bar{g})^{s-1} f_{2s-1} \left( gf_{2s-1} \right)^{s-1}
cg (\bar{f}_{2s-1})^{s-1} \left( gf_{2s-1} \right)^{s-1} ae (\bar{c})^{2s+6}.
\end{align*}

Finally, we erase the $f_{2s-1}$ by $bg=gf_{2s-1}$, and we obtain
$$L_{3,2s}=afc \bar{g}^s (bg)^s c \bar{b}^{s-1} g (bg)^{s-1} ae \bar{c}^{2s+6}.$$

We will continue to modification.
In the STEP 4, we replace $\bar{g}\bar{b}$ by $\bar{h}$ and $af$ by $h$, then we obtain
$$L_4=hc \bar{g}^s h^s c \bar{b}^{s-1} g h^{s-1} ae \bar{c}^{2s+6}.$$

In the STEP 5, the modification of presentation does not effect $L_4$, then $L_4=L_5$.

In the STEP 6, we erase the generator $a$ by $a=h\bar{f}$, then
$$L_6=hc \bar{g}^s h^s c \bar{b}^{s-1} g h^s \bar{f}e \bar{c}^{2s+6}.$$
In the rest steps, referring the relators which are used to eliminate generators, we obtain the sequence of $L_i$'s as follows:
\begin{align*}
L_7&= hc \bar{g}^s h^s c (k\bar{l})^{s-1} g h^s \bar{f}e \bar{c}^{2s+6} \\
L_8&= hc \bar{g}^s h^s c (k\bar{l})^{s-1} g h^s e \bar{c}^{2s+7} \\
L_9&= hc \bar{g}^s h^s c \left( gc\bar{g}\bar{l} \right)^{s-1} g h^s e \bar{c}^{2s+7} \\
L_{10}&= h \bar{e}^s c h^s c \left( \bar{c}ec\bar{e}c\bar{l} \right)^{s-1} \bar{c} ec h^s e \bar{c}^{2s+7} \\
L_{11}&= \bar{c}^s hc h^s c \left( \bar{c}\bar{h}chc\bar{h}\bar{c}hc\bar{l} \right)^{s-1} \bar{c}\bar{h} chc h^{s-1} ch \bar{c}^{2s+7} \\
L_{12}&= \bar{c}^{s-1} lc l^s \left( \bar{l}clc\bar{l}\bar{c} \right)^{s-1} \bar{l}clc l^{s-1} cl \bar{c}^{2s+8}.
\end{align*}

Although $L_{12}$ is the presentation of a longitude in $G_{K_s}$ at which we are aiming,
it is complicated to use for our proof of main theorem.
By using the relator ${\rm r_{\infty}:}\; clc \bar{l}\bar{c} \bar{l}^s \bar{c}\bar{l} clc l^{s-1}$
in $G_{K_s}$, we will simplify $L_{12}$ as follows.

We first rotate the word ${\rm r_{\infty}}$ and then obtain
$c \bar{l}\bar{c} \bar{l}^s \bar{c}\bar{l} clc l^{s-1}cl$.
It is equivalent to 
$$\bar{l} clc l^{s-1}cl=c l^s cl \bar{c}$$
and we replace the part of $L_{12}$
which equivalent to the left hand side of it by the right hand side.
Then we obtain
$$L'=\bar{c}^{s-1} lc l^s \left( \bar{l}clc\bar{l}\bar{c} \right)^{s-1} c l^s cl (\bar{c})^{2s+9}.$$
Next we also obtain
$clc\bar{l}\bar{c}=\bar{l}^{s-1}\bar{c}\bar{l}\bar{c}lcl^s$
by rotating the word ${\rm r_{\infty}}$ and splitting it.
We add the element $\bar{l}$ from the left to both side of the formula, then
$\bar{l}clc\bar{l}\bar{c}=\bar{l}^s \bar{c} \bar{l} \bar{c} lcl^s$.
By using it we can obtain
\begin{align*}
\left( \bar{l}clc\bar{l}\bar{c} \right)^{s-1}
&= \left( \bar{l}^s\bar{c}\bar{l}\bar{c}lcl^s \right)^{s-1} \\
&= \bar{l}^s \bar{c} \bar{l} \bar{c}^{s-1} lc l^s
\end{align*}
Applying this to $L'$, we obtain
\begin{align*}
L'&=
\bar{c}^{s-1} lc l^s \left( \bar{l}^s \bar{c} \bar{l} \bar{c}^{s-1} lc l^s \right) c l^s cl \bar{c}^{2s+9} \\
&= \bar{c}^{2s-2} lc l^s c l^s cl \bar{c}^{2s+9}. 
\end{align*}

In summary, we obtain the following Proposition.

\begin{prop}\label{presentation_of_Ks}
Let $K_s$ be a $(-2,3,2s+1)$-type Pretzel knot {\rm ($s\geqq 3$)}.
Then the knot group of $K_s$ has a presentation
$$G_{K_s}=\langle c,l \mid clc \bar{l}\bar{c} \bar{l}^s \bar{c}\bar{l} clc l^{s-1} \rangle,$$
and an element which represents the meridian $M$ is $c$ and
an element of the longitude $L$ is
$\bar{c}^{2s-2} lc l^s c l^s cl \bar{c}^{2s+9}$.
\end{prop}

In fact,
we can easily verify that $L$ corresponds to a preferred longitude,
which is a preimage of null-homological curve in $H_1(S^3 \setminus K_s)$
for the natural map $\varphi:\pi_1 \rightarrow \pi_1 / [\pi_1,\pi_1] \cong H_1(S^3\setminus K_s)$,
where $\pi_1=\pi_1(S^3 \setminus K_s)$.

\section{proof of main theorem}

Now we are ready to prove Theorem~\ref{maintheorem}.
We shall prove main theorem as an analogy of the proof of Jun~\cite{Jun} and Roberts, Shareshian, Stein~\cite{RSS}.
In this section
$G_{K_s}(p,q)$ denotes the fundamental group $\pi_1(E_{K_s}(p/q))$ of the closed manifold $E_{K_s}(p/q)$
obtained by Dehn surgery along $K_s$ with a slope $p/q$.
By Proposition~\ref{presentation_of_Ks},
we obtain a presentation of $G_{K_s}(p,q)$ as follows:
$$G_{K_s}(p,q)=\langle c,l \mid clc \bar{l}\bar{c} \bar{l}^s \bar{c}\bar{l} clc l^{s-1}, M^pL^q \rangle.$$

\begin{lemma}\label{lem1}
{\rm (}{cf. \cite[Lemma 8]{Jun}, \cite[Lemma 3.4]{RSS}}{\rm )}
There exists an element $k$ of $G_{K_s}(p,q)$ such that $M=k^q$ and $L=k^{-p}$.
\end{lemma}
\begin{proof}
By the property of a meridian and a longitude on a torus, the pair $M$ and $L$ satisfies $ML=LM$.
Since $p$ and $q$ are relatively prime,
there is a pair of integers $r$ and $s$ such that $rp+sq=1$.
Then $k=M^sL^{-r}$ satisfies the required condition.
\end{proof}

By Lemma~\ref{lem1}, we obtain the following fact:

\begin{fact}\label{fact1}
The element $k$ which appeared in Lemma~\ref{lem1} satisfies
$$k^{p-(4s+7)q} = \bar{l}\bar{c}\bar{l}^s \bar{c}\bar{l}^s \bar{c}\bar{l}.$$
\end{fact}
\begin{proof}
By Lemma~\ref{lem1} and Proposition~\ref{presentation_of_Ks},
$k^{-p}=L=\bar{c}^{2s-2} lc l^s c l^s cl \bar{c}^{2s+9}$.
Then
\begin{align*}
k^p &= c^{2s+9} \bar{l}\bar{c}\bar{l}^s \bar{c} \bar{l}^s \bar{c}\bar{l} c^{2s-2} \\
&= M^{2s+9} \bar{l}\bar{c}\bar{l}^s \bar{c} \bar{l}^s \bar{c}\bar{l} M^{2s-2} \\
&= k^{(2s+9)q} \bar{l}\bar{c}\bar{l}^s \bar{c}\bar{l}^s \bar{c}\bar{l} k^{(2s-2)q} \\
\Longleftrightarrow\;\;& 
\bar{k}^{(2s+9)q} k^p \bar{k}^{(2s-2)q}=\bar{l}\bar{c}\bar{l}^s \bar{c}\bar{l}^s \bar{c}\bar{l} \\
\Longleftrightarrow\;\;&
k^{p-(4s+7)q}=\bar{l}\bar{c}\bar{l}^s \bar{c}\bar{l}^s \bar{c}\bar{l}.
\end{align*}
\end{proof}

We will consider actions of $G_{K_s}(p,q)$ on a leaf space $\mathcal{T}$.
Each action is regarded as a homomorphism $\Phi:G_{K_s}(p,q)\rightarrow {\rm Homeo}(\mathcal{T})$.
If there is no ambiguity, we write its image $\Phi(g)$ of an element $g\in G_{K_s}(p,q)$ by the same symbol $g$,
that is, for a point $x\in \mathcal{T}$, we write $\Phi(g)(x)=xg$.
Moreover, for any two elements $g_1$ and $g_2$,
we write $\Phi(g_2)(\Phi(g_1)(x))=xg_1 g_2$.

\begin{lemma}\label{lem2}
{\rm (}{cf. \cite[Lemma 15]{Jun}, \cite[Lemma 3.5]{RSS}}{\rm )}
Let $P$ be a partially ordered set.
We assume that $p-(4s+7)q\geqq 0$, $q>0$ and
$G_{K_s}(p,q)$ acts on $P$ preserving its order.
If there is an element $x$ of $P$ which satisfies one of the following conditions:
\begin{enumerate}
\item[(1)] $xk=x$ and $x$, $xl$ are related in $P$, or
\item[(2)] $xl=x$ and $x$, $xk$ are related in $P$,
\end{enumerate}
then $x$ is fixed by any element of $G_{K_s}(p,q)$.
\end{lemma}
Notice that two elements $x$ and $y$ in $P$ are related for the order $<$ on $P$
if they satisfy the one of properties $x<y$, $x>y$ or $x=y$.

\begin{proof}
We first assume the condition (1).
Since $x=xk=xk^q=xM=xc$ by Lemma~\ref{lem1} we obtain $xc=x$, that is, $c$ fixes the element $x$.
Under the assumption that $x$ and $xl$ are related,
if we assume $xl=x$ then $x$ is fixed by any element $g$ of $G_{K_s}(p,q)$
since $G_{K_s}(p,q)$ is generated by the two elements $c$ and $l$.
Therefore, we assume $x<xl$ by choosing an order on $P$.
It is equivalent to $x\bar{l}<x$ because all actions of $G_{K_s}(p,q)$ preserve an order of $P$.
Then we obtain $x\bar{l}^2<x$ by $x\bar{l}^2<x\bar{l}<x$.
Repeating this consideration, we obtain $x\bar{l}^s <x$ for $s>0$.
Since $xc=x$, it is equivalent to $x\bar{c}=x$.
Because of the condition $x\bar{l}<x$ we obtain $x\bar{l}\bar{c}<x$ by $x\bar{l}\bar{c}<x\bar{c}=x$.
Similarly we obtain $x\bar{l}\bar{c}\bar{l}^s\bar{c}\bar{l}^s\bar{c}<x$.
By Fact~\ref{fact1},
we obtain $k^{p-(4s+7)q}l=\bar{l}\bar{c}\bar{l}^s\bar{c}\bar{l}^s \bar{c}$.
Using this, we obtain
$$x > x\bar{l}\bar{c}\bar{l}^s\bar{c}\bar{l}^s \bar{c}=xk^{p-(4s+7)q}l=xl,$$
since $k$ fixes $x$.
It is contradict to $xl>x$,
then $x=xl$ and therefore $x$ is fixed by any element of $G_{K_s}(p,q)$.

Next we assume the condition (2).
If $xk=x$, we obtain $x=xk=xk^q=xM=xc$ and then
$x$ is fixed by any element of $G_{K_s}(p,q)$ similarly.
Then we assume $x<xk$ by choosing an order on $P$.
Since $q>0$ we obtain $xc>x$ by $x<xk<xk^q=xM=xc$.
The formula of Fact~\ref{fact1} is equivalent to
$\bar{k}^{p-(4s+7)q}=lcl^scl^scl$.
If $p-(4s+7)q>0$, 
we obtain $x>x\bar{k}>\cdots >x\bar{k}^{p-(4s+7)q}$
by the assumption $x<xk$.
By these formulae, we obtain
\begin{align*}
x&> xlcl^scl^scl \\
\Longleftrightarrow \;
x\bar{l}&> xlcl^s cl^s c
=xcl^s cl^s c \\
&> xl^s cl^s c
= xcl^s c \\
&> xl^s c
=xc.
\end{align*}
Since $x\bar{l}=x$ it means $x>xc$.
It contradicts $xc>x$.
Therefore $xc=x$ and then all elements of $G_{K_s}(p,q)$ fix $x$.
In the case that $p-(4s+7)q=0$,
we obtain $lcl^s cl^s cl=1$
by $1=k^{p-(4s+7)q}=\bar{l}\bar{c}\bar{l}^s \bar{c}\bar{l}^s \bar{c}\bar{l}$.
Using this formula, we also obtain
\begin{align*}
x&= xlcl^scl^scl \\
\Longleftrightarrow \;
x\bar{l}&=xlcl^s cl^s c = xcl^s cl^s c \\
&> xl^s cl^s c
= xcl^s c \\
&> xl^s c
=xc.
\end{align*}
It similarly contradicts $xc>x$.
\end{proof}

\begin{lemma}\label{lem3}
{\rm (}{cf. \cite[Lemma 16]{Jun}}{\rm )}
We assume $q>0$ and
all actions of $G_{K_s}(p,q)$ on $\mathbb{R}$ preserve an orientation of $\mathbb{R}$.
If $xk>x$ for any $x\in \mathbb{R}$,
then $xl>x$ for any $x\in \mathbb{R}$.
\end{lemma}

\begin{proof}
Under the assumption, we obtain $xc>x$ for any $x\in\mathbb{R}$ by
$$x<xk<xk^q = xM=xc.$$
Using the relator in $G_{K_s}(p,q)$ we obtain
\begin{align*}
& clc\bar{l}\bar{c}\bar{l}^s \bar{c}\bar{l}clcl^{s-1}=1 \\
\Longleftrightarrow \;\;
& \bar{l}\bar{c}\bar{l}^s \bar{c}\bar{l}clcl^{s-1} clc =1 \\
\Longleftrightarrow \;\;
& clcl^{s-1} clc = lcl^s cl.
\end{align*}
Since $xc>x$ for any $x\in \mathbb{R}$ and
using this formula, we obtain
$$xclcl^{s-1}clc = x lcl^s cl < xclcl^s clc,$$
because if we take $x'=xclcl^scl \in \mathbb{R}$ then we obtain
$xclcl^s clc =x'c>x' =xclcl^scl$.
Since any element of $G_{K_s}(p,q)$ preserves an orientation of $\mathbb{R}$
we obtain
\begin{align*}
xclcl^s clc &> xclcl^{s-1} clc \\
\Longleftrightarrow
xclcl^s &> xclcl^{s-1}
\end{align*}
for any $x\in\mathbb{R}$.

Since the element $clc\in G_{K_s}(p,q)$ is thought as a homeomorphism of $\mathbb{R}$
there is a point $x\in\mathbb{R}$ which satisfies $x'=xclc$
for any point $x'\in \mathbb{R}$.
Therefore,
for any point $x'\in \mathbb{R}$ we obtain
$$x'l^s =xclcl^s > xclcl^{s-1} =x'l^{s-1}.$$
Then it follows that
$xl>x$ for any $x \in \mathbb{R}$.
\end{proof}

\begin{prop}\label{prop4}
{\rm (}{cf. \cite[Proposition 17]{Jun}}{\rm )}
If $q>0$ and $p/q\geqq 4s+7$,
for any homomorphism $\Phi:G_{K_s}(p,q)\rightarrow {\rm Homeo}^+(\mathbb{R})$
there is a point $x \in \mathbb{R}$ such that
$x$ is fixed by any element of $\Phi(G_{K_s}(p,q))$.
\end{prop}

\begin{proof}
We assume that
there is no point $x\in \mathbb{R}$ which is fixed by any element of $\Phi(G_{K_s}(p,q))$,
and we shall see a contradiction.

By the above assumption, there is no point which satisfies the conclusion of Lemma~\ref{lem2}.
Then there is no point $x\in \mathbb{R}$ which satisfies the assumptions of Lemma~\ref{lem2}.
Now we think $P=\mathbb{R}$ in Lemma~\ref{lem2} and $\mathbb{R}$ has an ordinary order,
then any two points in $\mathbb{R}$ are related.
Therefore we can assume that $xk>x$ for any $x\in\mathbb{R}$ by determining an orientation on $\mathbb{R}$
because the element $k$ is thought as an orientation-preserving homeomorphism of $\mathbb{R}$.

By Lemma~\ref{lem3} under this assumption, $xl>x$ for any $x\in \mathbb{R}$.
By Fact~\ref{fact1},
$$k^{p-(4s+7)q}=\bar{l}\bar{c}\bar{l}^s \bar{c}\bar{l}^s \bar{c}\bar{l}
\Longleftrightarrow
k^{(4s+7)q-p}=lcl^s cl^s cl \;\;\;(*)
$$

\noindent{\bf Claim 1.} $xlcl^s cl^s cl > xc^3$ for any $x\in \mathbb{R}$. 

\noindent{\it Proof of Claim 1.}
Since $xl>x$ for any $x\in\mathbb{R}$,
we have that $xlc>xc$ and $xl^s>x$ for any $x\in \mathbb{R}$.
If $x'=xlc$, then
$xlcl^s=x'l^s>x'=xlc>xc$,
and then we obtain
$xlcl^sc > xc^2$ for any $x\in\mathbb{R}$.
Similarly,
by $xlcl^s cl^s =x'l^s > x'=xlcl^sc$,
we obtain
$xlcl^s cl^s c > xlcl^sc^2 > xc^3.$
Therefore we obtain
$$xlcl^s cl^s cl=x'l>x'=xlcl^s cl^s c > xc^3.$$

\noindent{\bf Claim 2.} $xc^3>x$ for any $x \in \mathbb{R}$. 

\noindent{\it Proof of Claim 2.}
By the assumption $xk>x$ and $q>0$, we have
$$x < xk < xk^q = xM =xc.$$
Then we have $xc>x$ and $xc^3>x$.

By using the formula $(*)$ and Claim 1 and 2, we obtain
$$
xk^{(4s+7)q-p}=xlcl^s cl^s cl
> xc^3
> x.
$$
Under the assumption $xk>x$,
we have $(4s+7)q-p>0$ by above formula.
This is equivalent to $p/q<4s+7$.
Then it contradicts the assumption $p/q\geqq 4s+7$.
\end{proof}

\begin{proof}[{\it Proof of Theorem~\ref{maintheorem}}]
If $E_{K_s}(p/q)$ contains an $\mathbb{R}$-covered foliation,
all actions of $G_{K_s}(p,q)$ to $\mathbb{R}$ have no global fixed point.
By ~\cite[Corollary 7]{Jun}, if $q$ is odd, we can assume
all actions of $G_{K_s}(p,q)$ correspond to orientation-preserving homeomorphisms.
Then by Proposition~\ref{prop4} we conclude that
if $q>0$, $q$ is odd and $p/q\geqq 4s+7$, there is a point of $\mathbb{R}$
which fixed by any element of $G_{K_s}(p,q)$.
Therefore, $E_{K_s}(p/q)$ does not contain an $\mathbb{R}$-covered foliation.
\end{proof}

\section{some related topics and problems}

We will discuss some related topics and problems in this section.

We first mention about future problems.
In Theorem~\ref{maintheorem} we explore the case when the leaf space $\mathcal{T}$ is homeomorphic to
$\mathbb{R}$.
It is the first problem that we extend Theorem~\ref{maintheorem} to the general case that a leaf space $\mathcal{T}$
is homeomorphic to a simply connected $1$-manifold which might not be a Hausdorff space
similar to Theorem~\ref{ThmJun1}.
In this case, the situation of a leaf space $\mathcal{T}$ is very complicated.
But we already obtained the explicit presentation of the fundamental group of $E_{K_s}(p/q)$,
we are going to investigate the action of the fundamental group to a leaf space
referring the discussions used in \cite{RSS} and \cite{Jun}.
One of other directions of investigations is that we extend Theorem~\ref{maintheorem}
to the case for $(-2,2r+1,2s+1)$-type Pretzel knot $K_{r,s}$ ($r\geqq 1$, $s\geqq 3$).
In order to calculate the fundamental group $\pi_1(S^3\setminus K_{r,s})$,
we have to extend our method mentioned in Section 2.
In Section 2.2, the difference between the calculation for $K_{s}$ and $K_{r,s}$ exists from step 4 to step 9.
In these steps, we decrease a number of steps for disentangling these
crossings which corresponds to the integer $2r+1$.
Then by iterating these operation we will finally get the figure in Figure~\ref{fig_G_9to12} (a).
We are going to formulate these operations and corresponding Tietze transformations,
and intend to prove the case of
$K_{r,s}$ which is expected for the similar proof in Section 3.

Next we discuss about some related topics.
We first discuss our result in the viewpoint of Dehn surgery on knots.
A knot $K$ in $S^3$ has a finite or cyclic surgery if the resultant manifold $E_K(p/q)$
obtained by a non-trivial Dehn surgery along $K$ with a slope $p/q$ has a property
that its fundamental group is finite or cyclic respectively.
Determining and classifying which knots and slopes have a finite or cyclic surgery are an interesting problem.
If $E_K(p/q)$ contains a Reebless foliation, we can conclude that $E_K(p/q)$ does not have a finite and cyclic surgery
by properties we mentioned in Section 1.
For example, Delman and Roberts showed that no alternating hyperbolic knot
admits a non-trivial finite and cyclic surgery 
by proving the existence of essential laminations~\cite{DR}.
Our Pretzel knots $K_s$ are in the class of a Montesinos knot.
In \cite{IJ}, Ichihara and Jong showed that
for a hyperbolic Montesinos knot $K$
if $K$ admits a non-trivial cyclic surgery it must be $(-2,3,7)$-pretzel knot and the surgery slope is $18$ or $19$,
and if $K$ admits a non-trivial acyclic finite surgery it must be $(-2,3,7)$-pretzel knot and the slope is $17$, or $(-2,3,9)$-pretzel knot and the slope is $22$ or $23$.
In contrast, by this theorem, infinitely many knots in the family of pretzel knot $\{K_s\}$
which appeared in Theorem~\ref{maintheorem}
do not admit cyclic or finite surgery.
Then we have following corollary directly.

\begin{coro}\label{colo1}
There are infinitely many pretzel knots which does not admit finite or cyclic surgery,
but they admit Dehn surgery which produces a closed manifold
which cannot contain an $\mathbb{R}$-covered foliation.
\end{coro}

We had expected that proving the existence of Reebless foliations, especially $\mathbb{R}$-covered foliations,
or essential laminations
is of use for determining and classifying a non-trivial finite or cyclic surgery on other hyperbolic knots in the same way as \cite{DR},
but Corollary~\ref{colo1} means that
in the case of pretzel knots, an $\mathbb{R}$-covered foliation is not of use for it.

Next we discuss our result in the viewpoint of a left-orderable group.
A group $G$ is left-orderable
if there exists a total ordering $<$ of the elements of $G$ which is left invariant,
meaning that for any elements $f$, $g$, $h$ of $G$, if $f<g$ then $hf<hg$.
It is known that
a countable group $G$ is left-orderable if and only if
there exists a faithful action of $G$ on $\mathbb{R}$,
that is,
there is no point of $\mathbb{R}$ which fixed by any element of $G$.
By this fact,
if a closed $3$-manifold $M$ contains an $\mathbb{R}$-covered foliation,
the fundamental group of $M$ is left-orderable.
The fundamental groups $G_{K_s}(p,q)$ which satisfy the assumptions of Theorem~\ref{maintheorem}
do not have a faithful action on $\mathbb{R}$ by Proposition~\ref{prop4}.
Therefore we conclude the following corollary:

\begin{coro}
Let $K_s$ be a $(-2,3,2s+1)$-type Pretzel knot in $S^3$ ($s\geqq 3$),
$G=G_{K_s}(p,q)$ denotes the fundamental group of the closed manifold
which obtained by Dehn surgery along $K_s$ with slope $p/q$.
If $q>0$, $p/q\geqq 4s+7$ and $p$ is odd,
$G$ is not left-orderable.
\end{coro}

Roberts and Shareshian generalize the properties of the fundamental groups treated in \cite{RSS}.
They present conditions when the fundamental groups of a closed manifold obtained by Dehn filling of a once punctured torus bundle is not right-orderable~\cite[Corollary 1.5]{RS}.
These are examples of hyperbolic $3$-manifolds which has non right-orderable fundamental groups. 

Clay and Watson showed the following theorem in their paper.

\begin{theorem}\label{ThmCW}\textnormal{(A.\,Clay, L.\,Watson, 2012, \cite[Theorem 4.5]{CW})} 
Let $K_m$ be a $(-2,3,2m+5)$-type Pretzel knot.
If $p/q>2m+15$ and $m\geqq 0$, the fundamental group $\pi_1(E_{K_m}(p/q))$ is not left-orderable.
\end{theorem}

By the fact mentioned before, these fundamental groups do not have a faithful action on $\mathbb{R}$,
then these $E_{K_m}(p/q)$ do not admit an $\mathbb{R}$-covered foliation.
Although the method of the proof of Theorem~\ref{ThmCW} is different from our strategy,
it concludes a stronger result than ours in the sense of an estimation of surgery slopes.
By the aspects getting from these results, there are many interaction between a study of $\mathbb{R}$-covered foliations
and a study of left-orderability of the fundamental group of a closed $3$-manifold,
so we think that these objects will be more interesting.


\begin{thebibliography}{99}
	\bibitem{BZ}
		G.\,Burde, H.\,Zieschang,
		{\it Knots},
		Second edition. de Gruyter Studies in Mathematics, 5. Walter de Gruyter \& Co., Berlin, 2003. xii+559 pp. 
	\bibitem{CC1}
		A.\,Candel, L.\,Conlon,
		{\it Foliations I},
		Graduate Studies in Mathematics, {\bf 23}, AMS, Providence, RI, 2000. xiv+402 pp.
	\bibitem{CC2}
		A.\,Candel, L.\,Conlon,
		{\it Foliations II},
		Graduate Studies in Mathematics {\bf 60}, AMS, Providence, RI, 2003. xiv+545 pp.
	\bibitem{CW}
		A.\,Clay, L.\,Watson,
		{\it Left-orderable fundamental groups and Dehn surgery},
		International Mathematics Research Notices (2012), rns129, 29 pages.
	\bibitem{DR}
		C.\,Delman, R.\,Roberts,
		{\it Alternating knots satisfy Strong Property P},
		Comment. Math. Helv. 74 (1999), no. 3, 376--397.
	\bibitem{Fen}
		S.\,R.\,Fenley,
		{\it Laminar free hyperbolic $3$-manifolds},
		Comment. Math. Helv. {\bf 82}(2007), 247--321.
	\bibitem{GO}
		D.\,Gabai, U.\,Oertel,
		{\it Essential laminations in 3-manifolds}, 
		Ann. of Math. (2) {\bf 130} (1989), no. 1, 41--73. 
	\bibitem{HTT}
		H.\,M.\,Hilden, D.\,M.\,Tejada, M.\,M.\,Toro,
		{\it Tunnel number one knots have palindrome presentations},
		J. Knot Theory Ramifications, {\bf 11} (2002), no.5, 815--831.
	\bibitem{IJ}
		K.\,Ichihara, I.\,D.\,Jong,
		{\it Cyclic and finite surgeries on Montesinos knots}, 
		Algebr. Geom. Topol. {\bf 9} (2009), no. 2, 731--742.
	\bibitem{Jun}
		J.\,Jun,
		{\it $(-2,3,7)$-pretzel knot and Reebless foliation},
		Topology Appl. {\bf 145} (2004), no. 1-3, 209--232.
	\bibitem{Kob}
		T.\,Kobayashi,
		{\it A criterion for detecting inequivalent tunnels for a knot},
		Math. Proc. Cambridge Philos. Soc. {\bf 107} (1990), 483--491.
	\bibitem{MSY}
		K.\,Morimoto, M.\,Sakuma, Y.\,Yokota,
		{\it Identifying tunnel number one knots},
		J. Math. Soc. Japan {\bf 48} (1996), no.4, 667--688.
	\bibitem{No}
		S.\,Novikov,
		{\it Topology of foliations},
		Trans. Moscow Math. Soc. {\bf 14} (1965), 268--305.
	\bibitem{Pa}
		C.\,F.\,B.\,Palmeira,
		{\it Open manifolds foliated by planes},
		Ann. of Math. {\bf 107} (1978), 109--131.
	\bibitem{RSS}
		R.\,Roberts, J.\,Shareshian, M.\,Stein,
		{\it Infinitely many hyperbolic $3$-manifolds which contain no Reebless foliation},
		J. Amer. Math. Soc. {\bf 16} (2003), no.3, 639--679. 
	\bibitem{RS}
		R.\,Roberts, J.\,Shareshian,
		{\it Non-right-orderable 3-manifold groups},
		Canad. Math. Bull. {\bf 53} (2010), no. 4, 706--718. 
	\bibitem{Ro}
		H.\,Rosenberg,
		{\it Foliations by planes},
		Topology {\bf 7} (1968), 131--138.
	\bibitem{W}
		J.\,Weeks,
		SnapPea,
		\url{http://www.geometrygames.org/SnapPea/}
\end{thebibliography}
\end{document}